\documentclass[12pt]{amsart}
\usepackage{amscd,amssymb,verbatim,bbm,comment,color}

\theoremstyle{plain}
\newtheorem{thm}{Theorem}[section]
\newtheorem{cor}[thm]{Corollary}

\newtheorem{ex}[thm]{Example}

\newtheorem*{example*}{Example}

\numberwithin{equation}{section}

\newcommand{\abs}[1]{\lvert#1\rvert}

\title{A refined determinantal inequality for correlation matrices}

\author[N.~Gao]{Niushan Gao}
\address{Department of Mathematics, Ryerson University,  Toronto, Canada}
\email{niushan@ryerson.ca}

\author[A.~Kirillova]{Alexandra Kirillova}
\address{Department of Mathematics, University of Toronto, Toronto, Canada}
\email{alexandra.kirillova@mail.utoronto.ca}

\author[Z.~Tong]{Zihao Tong}
\address{Department of Mathematics, Ryerson University, Toronto, Canada}
\email{ztong@ryerson.ca}

\thanks{The first and second authors acknowledge support of an NSERC Discovery Grant.}



\date{\today}

\begin{document}

\maketitle

\begin{abstract}Olkin \cite{Olk14} obtained a neat upper bound for the determinant  of  a correlation matrix. In this note, we present an extension and improvement of his result.
\end{abstract}

\section{Introduction}
The theory of majorization has produced fruitful applications; see \cite{MOA11, Zhang11}. For two real vectors $x =(x_1, \ldots, x_n)$ and $y = (y_1, \ldots, y_n)$ ordered decreasingly, i.e., $x_1\ge \cdots \ge x_n$ and $y_1\ge \cdots \ge y_n$, we write $x\succ y$  and
 say that $x$ \emph{majorizes} $y$ if
\begin{eqnarray*}
 \sum_{i=1}^nx_i&=&\sum_{i=1}^ny_i,\\
  \sum_{i=1}^kx_i&\ge &\sum_{i=1}^ky_i, \ k = 1, \ldots , n-1.
  \end{eqnarray*}

Let $R = (r_{ij})$ be a given $n\times n$ correlation matrix.
Put $\widetilde{R}=(\widetilde{r}_{ij})$, where $\widetilde{r}_{ii}= 1, \widetilde{r}_{ij}=r_1$ for all $i\ne j$, and $r_1=\frac{\sum_{i\ne j}r_{ij}}{n(n-1)}$.
In his elegant note \cite{Olk14},  Olkin proved that
\begin{eqnarray*}\lambda(R)\succ\lambda(\widetilde{R}). \end{eqnarray*} As a consequence,
\begin{eqnarray}\label{olk}
 \det R\le \det \widetilde{R}=(1-r_1)^{n-1}\big(1+(n-1)r_1\big).
\end{eqnarray}
Here  $\lambda(A)=(\lambda_1(A), \ldots, \lambda_n(A))$ denotes the eigenvalues of an $n\times n$ matrix  $A$; they are ordered decreasingly whenever we talk about majorization.

In this article, we intend to extend and improve Olkin's result \eqref{olk}.
For this purpose, we need to employ the tool of variance majorization introduced in \cite{NW06}.
For two real vectors $x =(x_1, \ldots, x_n)$ and $y = (y_1, \ldots, y_n)$ ordered increasingly, i.e., $x_1\le \cdots \le x_n$ and $y_1\le \cdots \le y_n$, we write $x\stackrel{vm}{\succ} y$  and
 say that $x$ \emph{variance majorizes} $y$ if
\begin{eqnarray*}
\sum_{i=1}^nx_i&=&\sum_{i=1}^ny_i,\\
 \mathrm{Var}\big(x\big)&= &\mathrm{Var}\big(y\big),\\
 \mathrm{Var}\big(x[k]\big)&\ge &\mathrm{Var}\big(y[k]\big), \ k = 2, \ldots , n-1,
\end{eqnarray*}
where $x[k]=(x_1,\ldots,x_k)$ and $\mathrm{Var}$ stands for variance. Note that the eigenvalues of a matrix are always ordered increasingly when we talk about variance majorization.

For a given $n\times n$ correlation matrix $R=(r_{ij})$,  we construct two new matrices as follows. Put
$r_2=\sqrt{\frac{\sum_{i\ne j}|r_{ij}|^2}{n(n-1)}}$. Clearly, $\abs{r_1}\leq r_2\leq 1$. Let $\widehat{R}$ be the $n\times n$ matrix whose diagonal entries are all $1$ and off-diagonal entries are all $r_2$; let $\overline{R}$ be the $n\times n$ matrix whose diagonal entries are all $1$ and off-diagonal entries are all $-r_2$.

Our main result is the following observation.

\begin{thm}\label{thm1}
$\lambda(\overline{R})
\stackrel{vm}{\succ} \lambda(R) \stackrel{vm}{\succ}\lambda(\widehat{R})$.
\end{thm}
\begin{proof}
Direct computations show that the eigenvalues of $\widehat{R}$ are $1-r_2$ with multiplicity $n-1$ and $1+(n-1)r_2$ with multiplicity $1$.
The eigenvalues of $\overline{R}$ are $1+r_2$ with multiplicity $n-1$ and $1-(n-1)r_2$ with multiplicity $1$. Write them as $\lambda(\widehat{R})$ and $\lambda(\overline{R})$ (in increasing order), respectively.
It follows  that
\begin{align*}
\sum_1^n\lambda_i(\widehat{R})&=\sum_1^n\lambda_i(\overline{R})=\sum_1^n\lambda_i(R)=n,\\
\sum_{i=1}^n\lambda_i(\widehat{R})^2&=\sum_{i=1}^n\lambda_i(\overline{R})^2=\sum_{i=1}^n\lambda_i(R)^2=\sum_{i,j=1}^nr_{ij}^2=n+n(n-1)r_2^2.
\end{align*}
Thus $\lambda(\widehat{R})$, $\lambda(\overline{R})$, and $\lambda({R})$ all have the same mean and variance. Therefore, the desired result follows immediately from \cite[Lemma 2.1]{NW06}.
\end{proof}

\begin{cor}\label{cor2} $\det\overline{R} \le \det R\le \det \widehat{R}$.
   \end{cor}
\begin{proof}
The second inequality follows immediately by Theorem \ref{thm1} and \cite[Theorem 4.1]{NW06}.  If $1-(n-1)r_2\ge 0$, the first inequality follows for the same reasons. If $1-(n-1)r_2< 0$, then $\det \overline{R}<0\leq \det R $. \end{proof}

Note that $\det\widehat{R}=(1-r_2)^{n-1}\big(1+(n-1)r_2\big)$ and $\det \widetilde{R}=(1-r_1)^{n-1}\big(1+(n-1)r_1\big)$. The function $f(t)=(1-t)^{n-1}\big(1+(n-1)t\big)$ is increasing on $[-\frac{1}{n-1},0]$ and decreasing on $[0,1]$. Thus if $r_1\geq 0$, then $\det\widehat{R}\le \det\widetilde{R} $, and Corollary \ref{cor2} improves Olkin's \eqref{olk} in this case. However, for $r_1<0$,  $\det\widehat{R}> \det\widetilde{R} $ and  $\det\widehat{R}<\det\widetilde{R} $ both could possibly happen, depending on how close  $r_1$ is to $-\frac{1}{n-1}$.

\medskip
We give a few numerical examples to demonstrate the results.
\begin{ex}
\begin{itemize}
    \item[(1)]
$$R=\left(
  \begin{array}{ccc}
1&0&-0.5\\0&1&0.5\\-0.5&0.5&1\\
  \end{array}
\right); \quad \det R=0.5;$$
$$r_1=0,\;\det\widetilde{R}=1;\quad r_2= 0.4082,\;\det\widehat{R}=0.6361. $$
Clearly, $\det\widehat{R}$ is much smaller than $\det\widetilde{R}$. Slightly perturb the off-diagonal entries we can modify $R$ so that $r_1<0$ and $\det\widehat{R}< \det\widetilde{R} $.
\item[(2)] $$R=\left(
  \begin{array}{ccc}
    1&-0.3&-0.3\\
-0.3&1&-0.5\\
-0.3&-0.5&1\\
  \end{array}
\right); \quad \det R=0.48;$$
$$r_1=-0.3667,\;\det\widetilde{R}=0.4981;\quad r_2= 0.3786,\;\det\widehat{R}=0.6785. $$
Here, $\det\widehat{R}>\det\widetilde{R}$.
\end{itemize}
\end{ex}

\medskip

One may naturally define $r_p=\left(\frac{\sum_{i\ne j}\abs{r_{ij}}^p}{n(n-1)}\right)^{1/p}$, let $R_p$ be the $n\times n$ matrix whose diagonal entries are all $1$ and off-diagonal entries are all $r_p$, and wonder whether $$\det R\le \det R_p= (1-r_p)^{n-1}\big(1+(n-1)r_p\big)\eqno(*)$$ holds true for all $1< p<\infty$. For $1<p\leq 2$, since $0\leq  r_p\leq r_2$, $\det \widehat{R}\leq \det R_p$, so that by Theorem \ref{thm1}, $(*)$ holds true for such $p$. Thus, it is enough to ask whether $(*)$ holds true for all  $2< p<\infty$. Note that $\lim_{p\rightarrow\infty}r_p=r_\infty:=\max_{i\ne j}\abs{r_{ij}}$. Thus if this were to happen, we must have $$\det (R)\le (1-r_\infty)^{n-1}\big(1+(n-1)r_\infty\big).$$
The following example, however, rejects this.
\begin{ex}$$R=\left(
  \begin{array}{ccc}
1&0&0.8\\0&1&-0.5\\0.8&-0.5&1\\
  \end{array}
\right); \quad r_\infty=0.8;$$
$$\det R=0.11,\quad(1-r_\infty)^{2}\big(1+2r_\infty\big)=0.104. $$
\end{ex}


\begin{thebibliography}{90}



\bibitem{MOA11}
A.W.~Marshall, I.~Olkin, B.C.~Arnold, \emph{Inequalities: Theory of Majorization and its Applications}, 2nd edition, Springer, New York, 2011.

\bibitem{NW06}
M.G.~Neubauer, W.~Watkins, A Variance analog of majorization and some associated inequalities, \emph{Journal of Inequalities in Pure and Applied Mathematics} 7(3), Article 79, 2006.

\bibitem{Olk14}
I.~Olkin, A determinantal inequality for correlation matrices, \emph{Statistics and Probability Letters} 88, 88-90, 2014.

\bibitem{Zhang11}
F.~Zhang, \emph{Matrix Theory: Basic Results and Techniques}, 2nd edition, Springer, New York, 2011.


\end{thebibliography}
\end{document}